\documentclass[12pt]{amsart}
\usepackage{etex}
\usepackage{amsfonts, amsmath, amssymb, latexsym, mathtools}

\usepackage{amscd,amssymb}
\usepackage{verbatim}
\usepackage{pstricks}
\usepackage{pst-node}
\usepackage[mathscr]{euscript}
\usepackage{tikz}
\usetikzlibrary{matrix,arrows}
\usepackage[normalem]{ulem}
\usepackage{varwidth}
\usepackage{leftidx}
\usepackage{bm}

\newsymbol\pp 1275
\usepackage{tikz}
\usetikzlibrary{matrix,arrows,backgrounds,shapes.misc,shapes.geometric,patterns,calc,positioning}

\usepackage[colorlinks=true, pdfstartview=FitV, linkcolor=blue, citecolor=blue, urlcolor=blue]{hyperref}

\usepackage{fullpage}
\usepackage{graphicx}
\usepackage{enumerate}

\def\VR{\kern-\arraycolsep\strut\vrule &\kern-\arraycolsep}
\def\vr{\kern-\arraycolsep & \kern-\arraycolsep}

\newtheorem{theorem}{Theorem}

\newtheorem{lemma}[theorem]{Lemma}
\newtheorem{prop}[theorem]{Proposition}
\newtheorem{corollary}[theorem]{Corollary}

\theoremstyle{definition}

\newtheorem{rmk}[theorem]{Remark}
\newenvironment{remark}[1][]{\begin{rmk}[#1]\pushQED{\qed}}{\popQED \end{rmk}}

\newtheorem{qu}[theorem]{Question}

\newtheorem*{rmknonum}{Remark}

\newtheorem{obs}[theorem]{Observation}

\newtheorem{ex}[theorem]{Example}
\newenvironment{example}[1][]{\begin{ex}[#1]\pushQED{\qed}}{\popQED \end{ex}}

\newcommand{\tr}{\operatorname{Tr}}

\newcommand{\rep}{\operatorname{rep}}

\newcommand{\bl}{\operatorname{BL}}

\newcommand{\GL}{\operatorname{GL}}

\newcommand{\ZZ}{\mathbb Z}

\newcommand{\RR}{\mathbb R}

\newcommand{\QQ}{\mathbb Q}

\newcommand{\tup}{\mathbf p}

\newcommand{\locus}{\bm{\mathscr{V}}}
\newcommand{\Y}{Y}
\newcommand{\X}{\bm{\mathscr{X}}}
\newcommand{\V}{V}
\newcommand{\diag}{\mathsf{diag}}
\newcommand{\Q}{\mathcal{Q}}

\newcommand{\Id}{\mathbf{I}}

\newcommand{\ddim}{\operatorname{\mathbf{dim}}}
\newcommand{\cc}{\operatorname{\mathbf{c}}}
\newcommand{\dd}{\operatorname{\mathbf{d}}}

\newcommand{\ff}{\operatorname{\mathbf{r}}}

\newcommand{\nn}{\mathbf{n}}

\newcommand{\ar}{\mathscr{A}}
\newcommand{\s}{\mathcal{S}}

\newcommand{\capa}{\mathbf{cap}}

\newcommand{\pd}{\mathsf{PD}}

\newcommand\restr[2]{{
  \left.\kern-\nulldelimiterspace 
  #1 
  \vphantom{\big|} 
  \right|_{#2} 
  }}

\begin{document}
\title{Algebraicity of the Brascamp-Lieb constants}

\author{Calin Chindris and Harm Derksen}
\address{University of Missouri-Columbia, Mathematics Department, Columbia, MO, USA}
\email[Calin Chindris]{chindrisc@missouri.edu}

\address{Northeastern University, Boston, MA}
\email[Harm Derksen]{ha.derksen@northeastern.edu}

\date{\today}
\bibliographystyle{amsalpha}

\begin{abstract}
We show that the Brascamp-Lieb (BL) constant $\bl(-,\tup)$, where $\tup$ is a tuple of rational exponents, is a semi-algebraic function on the set of feasible data. Consequently, it is algebraic in the sense that it satisfies a polynomial relation of the form $P( \V, \bl(\V, \tup))=0$ for a non-zero polynomial $P$. In fact, we establish an analogous statement in the more general setting of quiver BL constants associated to representations of bipartite quivers.
\end{abstract}

\maketitle
\setcounter{tocdepth}{1}
\tableofcontents

\section{Introduction} 
Let $k$ and $m$ be positive integers. Let $\dd=(d_1, \ldots, d_k)$ and $\nn=(n_1, \ldots, n_m)$ be two tuples of positive integers and $\tup=(p_1, \ldots, p_m)$ a tuple of positive real numbers such that 
\begin{equation}
\sum_{i=1}^k d_i=\sum_{j=1}^m p_j n_j.
\end{equation}
Furthermore, for each pair $(i,j)\in [k]\times [m]$, let $\ar_{ij}$ be a finite index set. We can visualize the sets $[k]$, $[m]$, and $\ar_{ij}$, $(i,j)\in [k]\times [m]$, as the bipartite quiver $\Q$ with source vertices $\{\underline{1}, \ldots, \underline{k}\}$, sink vertices $\{1, \ldots, m\}$, and with the arrows from $\underline{i}$ to $j$ labeled by the elements of $\ar_{ij}$. 

Now, let 
$$
\V:=\left( \V_a \in \RR^{n_j \times d_i} \mid a \in \ar_{ij}, i \in [k], j\in [m] \right)
$$
be a tuple of matrices. We also view $\V$ as a representation of $\Q$ and refer to $(\V, \tup)$ as a quiver datum. Following \cite{ChiDer-2021}, we define the \emph{capacity} of $(\V, \tup)$ as
\begin{equation} \label{capacity-eqn}
\capa_\Q(\V,\tup)=\inf \left \{ {\prod_{i=1}^k \det\left( \sum_{j=1}^m p_j \left(\sum_{a \in \ar_{ij}}V_a^{\top}\cdot Y_j \cdot V_a \right) \right) \over \prod_{j=1}^m \det(Y_j)^{p_j}} \;\middle|\; Y_j \in \pd_{n_j} \right \},
\end{equation}
where $\pd_{n_j}$ denotes the set of all real positive definite $n_j\times n_j$ matrices.

The \emph{quiver BL constant} associated to $(\V, \tup)$ is defined as
\begin{equation}\label{BL-quiver-eqn}
\bl_\Q(\V, \tup)={1 \over \sqrt{\capa_\Q(\V, \tup)}}.
\end{equation}
We point out that when $k=1$ and  all $\ar_{1j}$ are singletons, $\bl_\Q(\V, \tup)$ is the celebrated BL constant associated to $(\V, \tup)$; see \cite{BenCarChrTao-2008} for a systematic study of those constants. When all $\ar_{ij}$ are singletons but $k$ is arbitrary, the logarithm of $\bl_{\Q}(\V, \tup)$ coincides  with the optimal constant in the Anantharam-Jog-Nair inequality associated to $(\V, \tup)$; see \cite{AJN-2022} and \cite{ChiDer-2023}.  

We say that $(\V, \tup)$ is \emph{feasible} if $\capa_\Q(\V, \tup)>0$ or, equivalently, $\bl_\Q(\V, \tup)< \infty$. We say that $(\V, \tup)$ is \emph{extremizable} if it is feasible and the minimum in $(\ref{capacity-eqn})$ is attained for some $Y=(Y_1, \ldots, Y_m) \in \prod_{j=1}^m \pd_{n_j}$. We call any such tuple $Y$ an \emph{extremizer} for $(\V, \tup)$.  

The BL constants (classical or quiver generalization) are very difficult to compute with only a very few explicit examples known (see for example \cite[Section 4]{BennettBezCowlingFlock2017}). In the classical case ($k=1$ and all $\ar_{1j}$ singletons), the Brascamp-Lieb constant restricted to the set of feasible data is known to be continuous; moreover, as briefly remarked in \cite[pp. 3]{Ben-Bez-Flo-Lee-2018}, it is analytic when restricted further to the set of simple data. For general bipartite quivers $\Q$ and exponents $\tup \in \RR^m_{\geq 0}$, it is proved in \cite[Theorem 3.6]{BezGauTsu-2026} that $\capa_\Q(-, \tup)$ is a locally H{\"o}lder continuous function on the set of feasible data.

In this paper, we study algebraic properties of $\capa_\Q(-, \tup)$ for rational $\tup \in \QQ^m_{>0}$, confirming a conjecture of the second-named author.

\begin{theorem}\label{thm:main} Assume that $\tup \in \QQ^m_{>0}$, and let
$$
\s:=\left \{ \V \in \prod_{(i,j) \in [k] \times [m]} ~  \prod_{a \in \ar_{ij}} \RR^{n_j \times d_i} \;\middle|\; (\V, \tup) \text{ is feasible} \right\}.
$$
Then $\capa_\Q(-, \tup): \s \to \RR$ is a semi-algebraic function. Consequently, it is an algebraic function in the sense that there exists a non-zero polynomial $P$ in $1+\sum_{i,j}|\ar_{ij}|d_i n_j$ variables such that 
$$
P(\V,\capa_\Q(\V, \tup))=0 \qquad \text{for all } \V \in \s.
$$
In particular, the classical BL constant is a semi-algebraic, hence algebraic, function on the set of feasible data.
\end{theorem}

We prove Theorem \ref{thm:main} in two steps. First, we show in Theorem \ref{thm:main-thm-cap-fiber} that for each extremizable datum $(\V, \tup)$, the capacity of $(\V,\tup)$ can be written as a ratio of determinant expressions evaluated at any point in the fiber over $\V$ of a semi-algebraic family. This result allows us to use the Definable Choice Property from real algebraic geometry. Under the rationality assumption on $\tup$, this leads to the semi-algebraicity of $\capa_\Q(-, \tup)$ on the set of extremizable data, as proved in Theorem \ref{thm:main-2}. In the final step,  we prove in Theorem \ref{thm:partition-semi-alg} that for each feasible datum $(\V, \tup)$, the computation of $\capa_\Q(\V, \tup)$ reduces to evaluating the capacity at any extremizable datum lying in the fiber over $\V$ of a second semi-algebraic family.  Together with Theorem \ref{thm:main-2},  this description of the capacity yields the desired result.

We point out that Theorem \ref{thm:main} does not hold without the rationality assumption on $\tup$ (see Example \ref{non-algebraic-example} for details).

\medskip
\noindent\textbf{Terminology.} A subset of $\RR^D$ is called \emph{semi-algebraic} if it is a finite union of sets defined by finitely many polynomial equalities and strict inequalities. A map between semi-algebraic sets is called semi-algebraic if its graph is semi-algebraic. A function $F:X\to \RR$, with $X\subseteq \RR^D$, is called \emph{algebraic} if there exists a nonzero polynomial $P(x,t)\in \RR[x_1,\dots,x_D,t]$ such that $P(x,F(x))=0$ for all $x\in X$. It is well known that every semi-algebraic function $F:X\to \mathbb{R}$ is algebraic; see the proof of Theorem \ref{thm:main-2} for an argument in our setting, which applies verbatim to any semi-algebraic function.

\section{Computing the capacity along semi-algebraic families}

We fix $\dd=(d_1, \ldots, d_k, n_1, \ldots, n_m) \in \ZZ^{k+m}_{>0}$ and $\tup=(p_1, \ldots, p_m) \in \RR^m_{>0}$ such that $\sum_{i=1}^k d_i=\sum_{j=1}^m p_j n_j$. Although our main theorem requires the exponents $p_j$ be rational, in this section we allow them to be arbitrary positive real numbers.

\subsection{Geometric data} We recall from \cite{ChiDer-2021} that $(\V, \tup)$ is a \emph{geometric quiver BL datum} if 

\begin{equation}\label{eq:geom1}
\sum_{j=1}^{m} p_j \sum_{a\in \ar_{ij}} \V_a^{\top} \V_a = I_{d_i},
\qquad \forall\, i\in [k],
\end{equation}
and
\begin{equation}\label{eq:geom2}
\sum_{i=1}^{k}\ \sum_{a\in \ar_{ij}} \V_a \V_a^{\top} = I_{n_j},
\qquad \forall\, j\in [m].
\end{equation}

We point out that when $\tup \in \QQ_{>0}^m$, the capacity $\capa_\Q(\V, \tup)$ is related to the capacity of a certain completely positive operator via \cite{ChiDer-2021}. Combined with \cite[Prop. 2.8 and Lemma 3.4]{GarGurOliWig-2020}, this proves that the capacity of a geometric datum is equal to one in the rational case. Here we explain how to extend this result to arbitrary weights $\tup \in \RR^{m}_{>0}$ using only elementary arguments. 

\begin{theorem}\label{thm:geom-cap-one}
If $(\V,\tup)$ is a geometric quiver BL datum, then
\[
\capa_\Q(\V,\tup)=1.
\]
\end{theorem}

\begin{proof}
We clearly have that $\capa_\Q(\V,\tup)\le 1$, since $(\V,\tup)$ is geometric. It remains to prove the reverse inequality. For any $Y=(Y_1,\dots,Y_m)\in \prod_{j=1}^m \pd_{n_j}$, we will show that
\begin{equation}\label{eq:main-ineq}
\prod_{i=1}^{k} \det\!\Bigg(\sum_{j=1}^{m} p_j
\Big(\sum_{a\in \ar_{i,j}} \V_a^{\top}Y_j \V_a\Big)\Bigg)
\ \ge\ 
\prod_{j=1}^{m} \det(Y_j)^{p_j}.
\end{equation}

We distinguish two cases.

\medskip
\noindent\textbf{Case 1:} $n_j=1$ for all $j\in [m]$.
In this case, we write $t_j>0$ for $Y_j$ and consider
\begin{equation}\label{eq:Rdef}
R(t_1,\dots,t_m):=\sum_{j=1}^{m} t_j R_j.
\end{equation}
where
\[
R_j=\diag\!\Big(p_j\!\!\sum_{a\in \ar_{1j}} \V_a^{\top} \V_a,\ \dots,\ 
p_j\!\!\sum_{a\in \ar_{kj}} \V_a^{\top} \V_a\Big),
\qquad \text{for every } j\in [m].
\]
Then $R(t_1,\dots,t_m)\in \mathbb{R}^{D\times D}$ is positive semi-definite,  where $D:=d_1+\cdots+d_k$, and its determinant is precisely the left-hand side of $(\ref{eq:main-ineq})$. Next consider the spectral decomposition of $R(t_1,\dots,t_m)$:
\begin{equation}\label{eq:spectralR}
R(t_1,\dots,t_m)=\sum_{\ell=1}^{D} \sigma_\ell\, u_\ell u_\ell^{\top},
\end{equation}
where $\sigma_1,\dots,\sigma_D$ are the eigenvalues of $R(t_1,\dots,t_m)$ and $\{u_1,\dots,u_D\}$ is an orthonormal basis of $\mathbb{R}^D$. From \eqref{eq:Rdef} and \eqref{eq:spectralR} we get
\[
\sigma_\ell=\sum_{j=1}^{m} t_j a_{\ell j},
\qquad\text{where}\qquad
a_{\ell j}:=u_\ell^{\top}R_j u_\ell\ge 0,
\]
since each $R_j$ is positive semidefinite. We also have
\begin{equation}\label{eq:a-sum1}
\begin{aligned}
\sum_{j=1}^{m} a_{\ell j}
&=u_\ell^{\top}\Big(\sum_{j=1}^{m} R_j\Big)u_\ell
\overset{\eqref{eq:geom1}}{=}
u_\ell^{\top}\diag(I_{d_1},\dots,I_{d_k})u_\ell
=1,
\qquad \forall\, \ell\in [D],
\end{aligned}
\end{equation}
and
\begin{equation}\label{eq:a-sum2}
\begin{aligned}
\sum_{\ell=1}^{D} a_{\ell j}
&=\tr(R_j)
\overset{\eqref{eq:geom2}}{=}
p_j,
\qquad \forall\, j\in [m].
\end{aligned}
\end{equation}
Finally, the left-hand side of \eqref{eq:main-ineq} equals
\[
\det\big(R(t_1,\dots,t_m)\big)
=\prod_{\ell=1}^{D} \sigma_\ell
=\prod_{\ell=1}^{D}\Big(\sum_{j=1}^{m} t_j a_{\ell j}\Big)
\ge \prod_{\ell=1}^{D}\prod_{j=1}^{m} t_j^{a_{\ell j}}
= \prod_{j=1}^{m} t_j^{\sum_{\ell=1}^{D} a_{\ell j}}
= \prod_{j=1}^{m} t_j^{p_j},
\]
where the inequality follows from \eqref{eq:a-sum1}, \eqref{eq:a-sum2}, and the weighted AM--GM inequality. This proves \eqref{eq:main-ineq} in Case~1.

\medskip
\noindent\textbf{Case 2:} $n_1, \ldots, n_m$ are arbitrary positive integers. In what follows, we explain how to reduce \eqref{eq:main-ineq} to Case~1. We begin by considering the spectral decomposition of each $Y_j$:
\[
Y_j=\sum_{\ell=1}^{n_j} t_{\ell j}\, u_{\ell j} u_{\ell j}^{\top},
\]
where $t_{\ell j}>0$, $\ell\in [n_j]$, are the eigenvalues of $Y_j$ and
$\{u_{\ell j}\}_{\ell \in [n_j]}$ forms an orthonormal basis of $\mathbb{R}^{n_j}$.
Then, for any $a\in \ar_{i,j}$ we can write
\begin{equation}\label{eq:expand}
\V_a^{\top}Y_j \V_a
=\sum_{\ell=1}^{n_j} t_{\ell j}\,\big(u_{\ell j}^{\top}\V_a\big)^{\top}\big(u_{\ell j}^{\top}\V_a\big).
\end{equation}

These considerations lead us to a new bipartite quiver $\widetilde{\Q}$ obtained from $\Q$ by replacing each sink vertex $j$ with $n_j$ new sinks $(1,j),\dots,(n_j,j)$. Moreover, the set of arrows from a source $i$ to a sink $(\ell,j)$ is $\{(\ell,a)\mid a\in \ar_{ij}\}$. Furthermore, $\V$ gives rise to a representation $\widetilde{\V}$ of $\widetilde{\Q}$ which is $\RR$ at each sink vertex $(\ell,j)$, and along an arrow $(\ell,a)$, with $\ell \in [n_j]$ and $a \in \ar_{ij}$,
\[
\widetilde{V}_{(\ell,a)}:=u_{\ell j}^{\top}V_a.
\]
Since $(\V, \tup)$ is a geometric datum and $\{u_{\ell j}\}_{\ell \in [n_j]} \subseteq \mathbb{R}^{n_j}$ is an orthonormal basis for every $j \in [m]$,  one immediately checks that $(\widetilde{V},\widetilde{p})$ is geometric as a quiver datum over $\widetilde{Q}$, where
\[
\widetilde{p}_{(\ell,j)}:=p_j,
\qquad \forall\, \ell\in [n_j],\ \forall\, j\in [m].
\]
Applying Case~1 to $(\widetilde{\V}, \widetilde{\tup})$ and using  \eqref{eq:expand}, we obtain
\[
\prod_{i=1}^{k} \det\!\Bigg(\sum_{j=1}^{m} p_j
\Big(\sum_{a\in \ar_{ij}} V_a^{\top}Y_jV_a\Big)\Bigg)
\ge
\prod_{j=1}^{m}\ \prod_{\ell=1}^{n_j} t_{\ell j}^{p_j}
=
\prod_{j=1}^{m} \det(Y_j)^{p_j}.
\]
This completes the proof.
\end{proof}

As a consequence of Theorem \ref{thm:geom-cap-one}, we obtain the following formula which plays an important role in the proof of Theorem \ref{thm:main-thm-cap-fiber}.

\begin{corollary}\label{coro:cap-geom-formula}
Let $(\V, \tup)$ be a quiver datum, and let $g_i\in \GL(d_i)$, $i \in [k]$, and $h_j \in \GL(n_j)$, $j \in [m]$, be such that 
\begin{equation}\label{eqn:conj-geom-datum}
\Big( \Big( h_j V_a g_i^{-1} \mid a \in \ar_{ij}, i \in [k], j\in [m]\Big), \tup \Big)
\end{equation} is a geometric datum. Then
\begin{equation}\label{eqn:capa-formula-transf-geom}
\capa_\Q(\V, \tup)={\prod_{i=1}^k \det(g_i)^2 \over \prod_{j=1}^m \det(h_j)^{2p_j}}.
\end{equation}
Furthermore, $(\V, \tup)$ is extermizable with extermizers $\left( h_j^\top h_j\right)_{j \in [m]}$.
\end{corollary} 

\begin{proof} Since $(\ref{eqn:conj-geom-datum})$ is a geometric datum, it follows from Theorem \ref{thm:geom-cap-one} that
\begin{align*}
1
&=\inf\left\{
\frac{
\displaystyle\prod_{i=1}^{k}
\det\!\left(
\sum_{j=1}^{m} p_j
\sum_{a\in \ar_{ij}}
g_i^{-\top}\,\V_a^{\top}\,h_j^{\top}\,Y_j\,h_j\,\V_a\,g_i^{-1}
\right)
}{
\displaystyle\prod_{j=1}^{m}\det(Y_j)^{p_j}
}
\;\middle|\;
Y_j\in \pd_{n_j}
\right\}.
\end{align*}
\noindent
Substituting $\widetilde Y_j$ for $h_j^{T}Y_j h_j$ in the infimum above,  we get
\begin{align*}
1
&=\inf\left\{
\frac{
\displaystyle
\left(\prod_{i=1}^{k}\det(g_i)^{-2}\right)
\left(\prod_{i=1}^{k}\det\!\left(
\sum_{j=1}^{m} p_j
\sum_{a\in \ar_{ij}}
\V_a^{\top}\,\widetilde Y_j\,\V_a
\right)\right)
}{
\displaystyle
\left(\prod_{j=1}^{m}\det(h_j)^{-2p_j}\right)
\left(\prod_{j=1}^{m}\det(\widetilde Y_j)^{p_j}\right)
}
\;\middle|\;
\widetilde Y_j\in \pd_{n_j}
\right\},
\end{align*}
\noindent
and so
\begin{equation*}
1
=
\frac{\displaystyle\prod_{i=1}^{k}\det(g_i)^{-2}}
{\displaystyle\prod_{j=1}^{m}\det(h_j)^{-2p_j}}
\;\capa_Q(\V,p).
\end{equation*}

\noindent
This now implies $(\ref{eqn:capa-formula-transf-geom})$. To prove the last part of the corollary, define $Y_j:=h_j^\top h_j \in \pd_{n_j}$ for every $j \in [m]$, and 
$$
M_i:=\sum_{j=1}^m p_j \left( \sum_{a \in \ar_{ij}} V_a^{\top} Y_j V_a \right) \qquad \text{for every } i \in [k].
$$
Then, using $(\ref{eq:geom1})$ and $(\ref{eqn:conj-geom-datum})$, we obtain that
$$
M_i=g_i^\top g_i \qquad \text{for every } i \in [k].
$$
Thus
$$
\capa_\Q(\V, \tup)=\frac{\prod_{i=1}^k \det(M_i)}{\prod_{j=1}^m \det(Y_j)^{p_j}},
$$
which proves that $\left( Y_j \right)_{j \in [m]}$ is an extremizer for $(\V, \tup)$.

\end{proof}

\subsection{Extremizable data} Our goal in this section is to show that the capacity of quiver data can be computed by selecting convenient points in the fibers of a semi-algebraic family. 

Set $D_1=\sum_{i,j} |\ar_{ij}| d_i n_j$ and $D_2=\sum_{j=1}^m n_j^2$. Let $\X \subseteq \RR^{D_1+D_2}$ be the semi-algebraic set consisting of all pairs $(\V, \Y)$ with $\V=(V_a) \in \prod_{(i, j) \in [k] \times [m]} ~ \prod_{a \in \ar_{ij}} \RR^{n_j \times d_i}$ and $\Y=(Y_j) \in \prod_{j=1}^m \pd_{n_j}$ such that
\begin{equation}\label{eqn:eqn1-defn-locus}
M_i:=\sum_{j=1}^m p_j \left( \sum_{a \in \ar_{ij}} V_a^{\top} Y_j V_a \right) \qquad \text{is inverible for all } i \in [k],
\end{equation}
and
\begin{equation}\label{eqn:eqn2-defn-locus}
\sum_{i=1}^k \sum_{a \in \ar_{ij}} V_a M_i^{-1} V_a^{\top}=Y_j^{-1}, \qquad \text{for all } j \in [m].
\end{equation}
Let $\pi_1:\RR^{D_1+D_2} \to \RR^{D_1}$ be the projection onto $\RR^{D_1}$. We are now ready to prove the following useful formula for the capacity.

\begin{theorem}\label{thm:main-thm-cap-fiber}
Let 
$$
\locus:=\left \{ \V \in \prod_{(i,j) \in [k] \times [m]} ~ \prod_{a \in \ar_{ij}} \RR^{n_j \times d_i} \;\middle|\; (\V, \tup) \text{~is extremizable} \right \}.
$$
Then
\begin{equation}\label{extre-semi-alg}
\locus=\pi_1(\X),
\end{equation}
and thus $\locus$ is a semi-algebraic set. Furthermore, for any $\V \in \locus$ and $\Y=(Y_j)_j \in \RR^{D_2}$ such that $(\V, \Y) \in \X$,
\begin{equation}\label{eq:cap-extremizers}
\capa_\Q(\V, \tup)={\prod_{i=1}^k \det(M_i) \over \prod_{j=1}^m \det(Y_j)^{p_j} }.
\end{equation}
\end{theorem}

\begin{remark}
When $\tup \in \QQ^n_{>0}$ is a rational weight, this was proved in \cite{ChiDer-2021} using Kempf-Ness theorem in invariant theory. But when $\tup$ has irrational coordinates, these tools are no longer available. Nonetheless, we give an elementary proof below that works in the general case. 
\end{remark}

\begin{proof}[Proof of Theorem \ref{thm:main-thm-cap-fiber}] To establish the inclusion $\locus \subseteq \pi_1(\X)$, one can use the same arguments as in the proof of the implication $``(\Longleftarrow)"$ of \cite[Theorem 20]{ChiDer-2021}.

It remains to show that for any $(\V, \Y) \in \X$, formula $(\ref{eq:cap-extremizers})$ holds. Set
\begin{equation}
g_i:=M_i^{1 \over 2} \in \GL(d_i) \text{ and } h_j:=Y_j^{1 \over 2} \in \GL(n_j),
\end{equation}
for all $i \in [k]$ and $j \in [m]$. Then it is straightforward to check that
\begin{equation}\label{claim:geom-data-from-extremizable}
\Big( \Big( h_j V_a g_i^{-1} \mid a \in \ar_{ij}, i \in [k], j\in [m]\Big), \tup \Big)
\end{equation}
is a geometric datum. Indeed,  since $\V$ and $Y$ satisfy $(\ref{eqn:eqn1-defn-locus})$ and $(\ref{eqn:eqn2-defn-locus})$, we obtain
\begin{align*}
\sum_{j=1}^{m}p_j\sum_{a\in \ar_{ij}}
g_i^{-\top}\,\V_a^{\top}\,h_j^{\top}Y_j h_j\,\V_a\,g_i^{-1}
&=
g_i^{-\top}\left(
\sum_{j=1}^{m}p_j\sum_{a\in \ar_{ij}}
\V_a^{\top}Y_j\V_a
\right)g_i^{-1}\notag\\
&=
g_i^{-\top} M_i\, g_i^{-1}=
M_i^{-1/2}\,M_i\,M_i^{-1/2}=\Id_{d_i} \qquad \text{for all } i \in [k].
\end{align*}

\noindent
Similarly, we have
\begin{align*}
\sum_{i=1}^k\sum_{a\in \ar_{ij}}
h_j\,\V_a\,g_i^{-1}\,g_i^{-\top}\,\V_a^{\top}\,h_j^{\top}
&=
h_j\left(
\sum_{i=1}^{k}\sum_{a\in \ar_{ij}}
\V_a\,M_i^{-1}\,\V_a^{\top}
\right)h_j^{\top}\notag\\
&=
h_j\,Y_j^{-1}\,h_j^{\top}=
Y_j^{1/2}\,Y_j^{-1}\,Y_j^{1/2}=\Id_{n_j} \qquad \text{for all } j \in [m].
\end{align*}
This shows that $(\ref{claim:geom-data-from-extremizable})$ is indeed a geometric datum. The desired formula $(\ref{eq:cap-extremizers})$ now follows from Corollary \ref{coro:cap-geom-formula}. 
\end{proof}

Let us record the following consequence of Corollary \ref{coro:cap-geom-formula} and Theorem \ref{thm:main-thm-cap-fiber}, which characterizes extremizable quiver data as precisely those that can be transformed into geometric data.

\begin{corollary}\label{coro:extremizable-geometric} Let $(\V, \tup)$ be a quiver datum. Then $(\V, \tup)$ is extremizable if and only if there exist transformations $g_i\in \GL(d_i)$, $i \in [k]$, and $h_j \in \GL(n_j)$, $j \in [m]$, such that $(W, \tup)$ is a geometric datum, where $W_a:=h_j V_a g_i^{-1}$ for all $a \in \ar_{ij}, i \in [k], j\in [m]$.
\end{corollary}

\section{Proof of Theorem \ref{thm:main}}\label{sec:proof-main-thm}

In this final section, we first establish the semi-algebraicity of the restriction of the capacity to the set of extremizable data.  We then use the quiver invariant theoretic approach developed in \cite{ChiDer-2021} to reduce the general case to the extremizable case.

\subsection{The algebraicity of the capacity restricted to extremizable data} Let $d_1, \ldots, d_k$ and $n_1, \ldots, n_m$ be positive dimensions and $\tup=(p_1, \ldots, p_m) \in \RR^m_{>0}$ a tuple of positive exponents such such $\sum_{i=1}^k d_i=\sum_{j=1}^m p_j n_j$. Recall that 
\[
\locus=\left \{ \V \in \prod_{(i,j) \in [k] \times [m]}~\prod_{a \in \ar_{ij}} \RR^{n_j \times d_i} \;\middle|\; (\V, \tup) \text{~is extremizable} \right \}.
\]
We know from Theorem \ref{thm:main-thm-cap-fiber} that $\locus=\pi_1(\X)$, where $\X \subseteq \RR^{D_1+D_2}$ is the semi-algebraic family consisting of all pairs $(\V, \Y)$ satisfying $(\ref{eqn:eqn1-defn-locus})$ and $(\ref{eqn:eqn2-defn-locus})$, and $\pi_1:\RR^{D_1+D_2} \to \RR^{D_1}$ is the projection onto $\RR^{D_1}$. We are now ready to prove out first algebraicity result.

\begin{theorem}\label{thm:main-2} Assume that $\tup \in \QQ_{>0}^m$. Then $\capa_\Q(-, \tup): \locus \to \RR$ is a semi-algebraic function. Consequently, it is an algebraic function in the sense that there exists a non-zero polynomial $P$ in $1+\sum_{i,j}|\ar_{ij}|d_i n_j$ variables such that 
$$
P(\V,\capa_\Q(\V, \tup))=0 \qquad \text{for all } \V \in \locus.
$$
\end{theorem} 

\begin{proof}
The Definable Choice Property \cite[Theorem 3.1]{Coste-1999}, applied to $\pi_1$ and the semi-algebraic family $\X$, yields a semi-algebraic function
\begin{equation*}
f:\locus \to \RR^{D_2},
\end{equation*}
such that $(\V, f(\V)) \in \X$ for all $V \in \locus$.

Next, taking $(Y_j)=f(\V)$ in $(\ref{eqn:eqn1-defn-locus})$ makes each $M_i$ a semi-algebraic function on $\locus$. Furthermore, since $\tup$ is assumed to be rational, the denominator in $(\ref{eq:cap-extremizers})$ is a semi-algebraic function, as well. Therefore, $\capa_\Q(-, \tup):\locus \to \RR$ is semi-algebraic. 

Finally, the Zariski closure of the graph of $\capa_\Q(-, \tup)$ has the same dimension as $\locus$, which is at most $D_1$, so it is an algebraic subset that is strictly contained in $\RR^{D_1+1}$. Thus there exists a non-zero polynomial $P$ in $D_1+1$ variables such that
$$
P(\V, \capa_\Q(\V, \tup))=0 \qquad \text{for all } \V \in \locus.
$$
\end{proof}

\begin{rmk}\label{non-rationality-rmk} Without the rationality assumption on $\tup$, the denominator in $(\ref{eq:cap-extremizers})$ is not semi-algebraic. In fact, that is the only place in the proof above where the rationality of $\tup$ is used. 
\end{rmk}

\begin{example}\label{non-algebraic-example} We are grateful to Neal Bez, Jon Bennett, and Shohei Nakamura for providing the following example. It shows that, if the exponents $p_1, \ldots, p_m$ are not rational, the capacity/BL constant is not semi-algebraic.

Choose $k=1$, $m=2$, and $d_1=n_1=n_2=1$. Fix $\tup=(p_1, p_2) \in \RR^2_{>0}$ such that $p_1+p_2=1$. Let $\V=(\V_1:\RR \to \RR, \V_2:\RR\to \RR)$ be a datum with
\[
\V_1(x)=ax,\qquad \V_2(x)=bx \qquad (x\in \mathbb R),
\]
where $(a,b)\in \RR^2$. Using the weighted AM-GM inequality, it is immediate to see that
\[
\capa_\Q(\V, \tup)=a^{2p_1}b^{2p_2},
\]
and thus $\bl(\V, \tup)=|a|^{-p_1}|b|^{-p_2}$. So the set of all feasible data is 
\[
\{(a,b)\in \mathbb R^2 \mid a\neq 0,\ b\neq 0\}.
\]
Furthermore, for any $(a,b)$ in the set of feasible data, the corresponding datum $\V$ is extremizable with extremizers given by $Y_1=b^2$ and $Y_2=a^2$.

But the function $(a,b) \to a^{2p_1}b^{2p_2}$, with $a\neq 0$ and $b \neq 0$, can not be semi-algebraic if the exponents are irrational numbers. Indeed, choose $p=p_1$ to be irrational and assume for a contradiction that it is semi-algebraic. In particular, the function $t \to t^{2p}$ is semi-algebraic. Thus there exists a non-zero polynomial $P \in \RR[x,y]$ such $P(t, t^{2p})=0$ for all $t\neq 0$. Writing $P(x,y)=\sum_{i,j \in \mathbb{Z}_{\geq 0}}c_{ij}x^i y^j$,  we get $P(t, t^{2p})=\sum_{i,j} c_{ij} t^{i+2pj}$. Since $p$ is irrational, the exponents $i+2pj$ are all distinct, and therefore the coefficients $c_{ij}$ must vanish (contradiction).
\end{example}

\subsection{The capacity as a function on the entire set of feasible data}
In this section, we adopt the quiver invariant approach to the study of BL constants as developed in \cite{ChiDer-2021} (see also \cite{ChiDer-2023}). 

For our general bipartite quiver $\Q$, we denote the set of vertices by $\Q_0:=\{\underline{1}, \ldots, \underline{k}\} \cup \{1, \ldots, m\}$ and the set of arrows by $\Q_1:=\cup_{(i,j) \in [k] \times [m]} \ar_{ij}$. For an arrow $a$, we write $ta \in \{\underline{1}, \ldots, \underline{k}\}$ for its \emph{tail} and $ha \in  \{1, \ldots, m\}$ for its \emph{head}. We also write $\Q^+_0$ for the set of source vertices $\{\underline{1}, \ldots, \underline{k}\}$ and $\Q^-_0$ for the set of sink vertices $\{1, \ldots, m\}$.

A \emph{representation} of $\Q$ is a family $W=(W_x,W_a)_{x \in \Q_0, a\in \Q_1}$, where $W_x$ is finite dimensional $\RR$-vector spaces for every $x \in \Q_0$, and $W_a: W_{ta} \to W_{ha}$ is an $\RR$-linear map for every $a \in \Q_1$.  After fixing bases for the vector spaces $W_x$, $x \in \Q_0$, we often view the linear maps $W_a$, $a \in \Q_1$, as matrices of appropriate size. The dimension vector of $W$ is $\ddim(W) \in \ZZ^{\Q_0}_{\geq 0}$ with $\ddim(W)_x:=\dim_\RR W_x$ for all $x \in \Q_0$.

Let $W$ be a representation of $\Q$. A subrepresentation of $W$ is a representation $W'$ such that $W'_x$ is a subspace of $W_x$ for every $x \in \Q_0$, $W_a(W'_{ta}) \subseteq W'_{ha}$, and $W'_a=\restr{W_a}{W'_{ta}}$ for every $a \in \Q_1$. The \emph{direct sum} of two representations is defined by taking direct sums at the level of vector spaces and linear maps.  The abelian category of all finite dimensional representations of $\Q$ is denoted by $\rep(\Q)$.

By a \emph{dimension vector} $\ff$ of $\Q$, we mean an assignment of non-negative integers to the vertices of $\Q$, \emph{i.e.}, $\ff \in \ZZ^{Q_0}_{\geq 0}$. For a tuple $\cc \in \RR^{m}$, we write $\ff \perp \cc$ to mean that 
$$
\sum_{x \in \Q^+_0} \ff_x=\sum_{y \in \Q^-_0} \ff_y \cc_y.
$$

For a given dimension vector $\ff=(\ff_x)_{x \in \Q_0}$ of $\Q$, the representation space of $\ff$-dimensional real representations of $\Q$ is
$$
\rep(\Q, \ff):=\prod_{a \in \Q_1} \RR^{\ff_{ha}\times \ff_{ta}}.
$$
The change-of-base group $\GL(\ff):=\prod_{x \in \Q_0} \GL(\ff_x)$ acts on $\rep(\Q,\ff)$ by \emph{simultaneous conjugation}: For $W=(W_a)_{a \in \Q_1} \in \rep(\Q,\ff)$ and $A=(A_x)_{x \in \Q_0}\in \GL(\ff)$, $A \cdot W \in \rep(\Q,\ff)$ is defined by
$$
(A \cdot W)_a:=A_{ha} W_a A_{ta}^{-1} \qquad \text{for every } a \in \Q_1.
$$

We record the following straightforward but useful lemma, whose proof we leave to the reader.

\begin{lemma}\label{lemma:quick-lemma} Let $W\in \rep(\Q, \ff)$, $A \in \GL(\ff)$, and $\cc \in \RR^{m}_{>0}$ with $\ff \perp \cc$. 
\begin{enumerate}[(a)]
\item The following formula holds:
\begin{equation}
\capa_Q(A \cdot W, \cc)=
{\prod_{x \in \Q^+_0} \det(A_x)^{-2} \over \prod_{y \in \Q^-_0} \det(A_y)^{-2p_y}} \cdot \capa_\Q(W,\cc).
\end{equation}
Consequently, for $W_1, W_2 \in \rep(\Q, \ff)$, 
\begin{equation}\label{eqn:quick-lemma}
\capa_\Q(A^{-1} \cdot W_1, \cc)=\capa(W_2, \cc) \Longleftrightarrow \capa_\Q(W_1, \cc)=\capa_\Q(A \cdot W_2, \cc).
\end{equation}
\item The property of being extremizable is invariant under the action of $A$: 
\begin{equation}
(W, \cc) \text{ is extremizable } \Longleftrightarrow (A\cdot W, \cc) \text{ is extremizable}
\end{equation}
\end{enumerate}
\end{lemma}

\bigskip
Now let $W$ be a representation of $\Q$ and $\cc \in \RR^m_{\geq 0}$. We say that $W$ is \emph{$\cc$-semi-stable} if
\begin{equation}\label{eqn:semi-stable-1}
\sum_{x \in \Q^+_0} \ddim(W)_x=\sum_{y \in \Q^-_0} \ddim(W)_y \cc_y, \textit{ i.e. } \ddim(W) \perp \cc,
\end{equation}
and
\begin{equation}\label{eqn:semi-stable-2}
\sum_{x \in \Q^+_0} \ddim(W')_x \leq \sum_{y \in \Q^-_0} \ddim(W')_y \cc_y,
\end{equation}
for every subrepresentation $W'$ of $W$. We say that $W$ is \emph{$\cc$-stable} if $W\neq 0$, $\ddim(W) \perp \cc$, and $(\ref{eqn:semi-stable-2})$ holds with strict inequality for all proper subrepresentations $0 \neq W' \subsetneq W$.

Let $\rep(\Q)^{ss}_{\cc}$ be the full subcategory of $\rep(\Q)$ whose objects are the $\cc$-semi-stable representations of $\Q$. This is an abelian subcategory of $\rep(\Q)$, closed under extensions. Its simple objects are precisely the $\cc$-stable representations of $\Q$. Furthermore, it is Artinian and Noetherian; hence every $\V \in \rep(\Q)^{ss}_{\cc}$ has a Jordan-H{\"o}lder filtration whose factors are $\cc$-stable.
 
For the remainder of the paper, we fix the following data:
\begin{itemize}
\item a dimension vector $\dd=(d_1, \ldots, d_k, n_1, \ldots, n_m) \in \ZZ^{\Q_0}_{> 0}$;
\item a rational tuple $\tup \in \QQ^m_{>0}$ such that $\dd \perp \tup$. 
\end{itemize}

The rationality assumption on $\tup$ serves two purposes. First, it ensures that the relevant powers of semi-algebraic functions are again semi-algebraic; see Remark \ref{non-rationality-rmk} and Example \ref{non-algebraic-example}. Second, it allows one to bring in methods from quiver invariant theory, such as the Kempf--Ness theorem on closed orbits. In particular, the results from \cite{ChiDer-2021, AJN-2022} that we summarize below rely on such tools.

\begin{theorem}\label{thm:semi-stable-summary} For $\V \in \rep(\Q,\dd)$, the following statements hold.
\begin{enumerate}
\item $(\V, \tup)$ is a feasible BL datum if and only if $\V$ is $\tup$-semi-stable. 
\smallskip
\item $(\V, \tup)$ is extremizable if and only if $\V$ is a direct sum of $\tup$-stable representations. 
\end{enumerate}
\end{theorem}

In what follows, for a representation $W \in \rep(\Q, \dd)$ and a tuple of dimension vectors $(\dd^s, \ldots, \dd^1)$ with $\sum_{\ell=1}^s \dd^\ell=\dd$, we say that $W$ is \emph{upper-triangular of type $(\dd^s, \ldots, \dd^1)$} if for each arrow $a \in \Q_1$, the matrix $W_a$ has a block upper triangular structure with the block diagonal entries having sizes $\dd^\ell_{ha}\times \dd^\ell_{ta}$, $\ell \in [s]$. For such a representation, set $\diag(W)_a$ to be the direct sum of the block diagonal entries of $W_a$ for every $a \in \Q_1$, and define $\diag(W):=(\diag(W)_a)_{a \in \Q_1} \in \rep(\Q, \dd)$.

Our next result plays an essential role in the proof of Theorem \ref{thm:partition-semi-alg}. The latter is one of the main ingredients in the proof of Theorem \ref{thm:main}.  

\begin{prop}\label{prop:main} Let $(\V, \tup) \in \s$ be a feasible datum with $\V \in \rep(\Q, \dd)$. Then there exist $A \in \GL(\dd)$ and dimension vectors $\dd^1, \ldots, \dd^s$ such that
\begin{enumerate}
\item $\dd^1+\ldots+\dd^s=\dd$ and $\dd^\ell \perp \tup$ for every $\ell \in [s]$;
\smallskip
\item $A \cdot \V$ is upper-triangular of type $(\dd^s, \ldots, \dd^1)$;
\smallskip
\item $(\widetilde{\V}, \tup)$ is extremizable where $\widetilde{V}:=\diag(A \cdot \V)$.
\end{enumerate}
\end{prop}

\begin{proof} Since $(\V,\tup)$ is a feasible datum, $\V$ is $\tup$-semi-stable by Theorem \ref{thm:semi-stable-summary}{(1)}. Now let us consider the Jordan--H\"older filtration of $\V$ in the category of all
$\tup$-semi-stable representations:
\[
0=\V_s\subset \V_{s-1}\subset \cdots \subset \V_1 \subset \V_0=\V
\]
where $\V_{\ell-1}/\V_\ell$ is $\tup$-stable for every $\ell\in [s]$. Set $\dd^\ell:=\ddim(\V_{\ell-1}/\V_\ell)$ for every $\ell \in [s]$.

After choosing a basis for each $\V_x=\RR^{d_x}$ compatible with this filtration, we obtain a transformation
$A \in \GL(\dd)$ such that, for every arrow $a \in \Q_1$,
\[
(A\cdot \V)_a =
\begin{pmatrix}
X_a^s & *      & \cdots & * \\
0     & X_a^2  & \cdots & * \\
\vdots& \vdots & \ddots & \vdots \\
0     & 0      & \cdots & X_a^1
\end{pmatrix},
\]
where $X^\ell := (X_a^\ell)_{a\in \Q_1} \in \rep(\Q,\dd^\ell)$ is isomorphic to the $\tup$-stable factor $\V_{\ell-1}/\V_\ell$ for every $\ell \in [s]$. This shows that $A\cdot \V$ is block upper-triangular
of type $(\dd^s,\dots,\dd^1)$; furthermore, since $\diag(A\cdot \V)$ is a direct sum of $\tup$-stable
representations, $(\widetilde{\V}, \tup)$ is extremizable by Theorem \ref{thm:semi-stable-summary}{(2)}.
\end{proof}

Next, we use Proposition \ref{prop:main} in an essential way to show that the computation of $\capa_\Q(\V, \tup)$ reduces to evaluating the capacity at an extremizable datum lying in the fiber over $\V$ of a semi-algebraic family.

In what follows, for the dimension vector $\dd=(d_1, \ldots, d_k, n_1, \ldots, n_m) \in \ZZ^{\Q_0}_{> 0}$, we have $\rep(\Q, \dd)=\RR^{D_1}$ where $D_1=\sum_{i,j} |\ar_{ij}| d_i n_j$. Also, recall that
$$
\s=\left \{ \V \in \RR^{D_1} \;\middle|\; (\V, \tup) \text{ is feasible} \right\}.
$$

\begin{theorem}\label{thm:partition-semi-alg} There exists a semi-algebraic family $\mathcal{Y} \subseteq \RR^{D_1}\times \RR^{D_1}$ such that 
\begin{enumerate}
\item $\s =\pi_1(\mathcal{Y})$, where $\pi_1:\RR^{D_1}\times \RR^{D_1} \to \RR^{D_1}$ is the first projection; 

\item for every $(\V, \V') \in \mathcal{Y}$,  $(\V', \tup)$ is extremizable and
\begin{equation}\label{eq:capa-reduction-extremizable}
\capa_\Q(\V, \tup)=\capa_\Q(\V', \tup).
\end{equation}
\end{enumerate}
\end{theorem}

\begin{proof} For every tuple of dimension vectors $(\dd^s, \ldots, \dd^1)$ satisfying Proposition \ref{prop:main}{(1)}, consider the subset of $\RR^{D_1}\times \RR^{D_1} \times \GL(\dd)$ consisting of all triples $(\V, \V', A)$ such that:
\begin{enumerate}[(i)]
\item $(\V', \tup)$ is extremizable; 
\item $A\cdot \V$ is upper triangular of type $(\dd^s, \ldots, \dd^1)$; 
\item $A\cdot \V'=\diag(A \cdot \V)$. 
\end{enumerate}
This is a semi-algebraic set since (i) is a semi-algebraic condition by Theorem \ref{thm:main-thm-cap-fiber}, and (ii) and (iii) are clearly algebraic conditions. Therefore, the union of all such subsets, denoted by $\mathcal{Z}$, is also a semi-algebraic subset of $\RR^{D_1}\times \RR^{D_1} \times \GL(\dd)$. 

Now let $\mathcal{Y}$ denote the projection of $\mathcal{Z}$ onto $\RR^{D_1}\times \RR^{D_1}$. Then $\mathcal{Y}$ is a semi-algebraic set. Furthermore,  for any $\V \in \s$, we have that $\V=\pi_1(\V, \V')$, where $\V':=A^{-1} \cdot \diag(A \cdot \V)$ and $A \in \GL(\dd)$ is as in Proposition \ref{prop:main}. Thus
$$
\s=\pi_1(\mathcal{Y}). 
$$

It remains to show that $(\ref{eq:capa-reduction-extremizable})$ holds.  Let $(\V, \V') \in \mathcal{Y}$ and $A \in \GL(\dd)$ be such that $(\V, \V', A) \in \mathcal{Z}$. For every arrow $a \in \Q_1$, write
\[
(A\cdot \V)_a =
\begin{pmatrix}
X_a^s & *      & \cdots & * \\
0     & X_a^2  & \cdots & * \\
\vdots& \vdots & \ddots & \vdots \\
0     & 0      & \cdots & X_a^1
\end{pmatrix},
\]
where $X^\ell := (X_a^\ell)_{a\in \Q_1} \in \rep(\Q,\dd^\ell)$ for every $\ell \in [s]$. Next, we are going to degenerate $A\cdot \V$ to $\diag(A\cdot \V)$ via a convenient one-parameter subgroup. Concretely, let
\[
\lambda:\RR^{\ast}\to \GL(\dd)
\]
be the one-parameter subgroup defined by
\[
\lambda(t)_x =
\begin{pmatrix}
t^{s-1} \Id_{d_x^{s}} & 0 & \cdots & 0 \\
0 & t^{s-2} \Id_{d_x^{s-1}} & \cdots & 0 \\
\vdots & \vdots & \ddots & \vdots \\
0 & 0 & \cdots & \Id_{d_x^1}
\end{pmatrix},
\qquad x\in \Q_0.
\]

Then $(\lambda(t)\cdot (A\cdot \V))_a$ is block upper-triangular, whose $(i,j)$-block
entry is the $(i,j)$-block entry of $(A\cdot \V)_a$ multiplied by $t^{\,j-i}$.
Therefore,
\begin{equation} \label{eq:limit-1-psg}
\lim_{t\to 0}\lambda(t)\cdot (A\cdot \V)
=
\operatorname{diag}(A\cdot \V).
\end{equation}

Next, we claim that
\begin{equation}\label{eqn:capa-1-psg}
\capa_\Q(\lambda(t)\cdot (A\cdot \V),\tup)
=
\capa_\Q(A\cdot \V,\tup),
\qquad \forall t\in \RR^\ast.
\end{equation}
Indeed,  using the fact that $\dd_i \perp \tup$ for every $i \in [s]$, it is immediate to see that
\[
\prod_{x\in \Q_0^+}\det(\lambda(t)_x)^2
=
\prod_{y\in \Q_0^-}\det(\lambda(t)_y)^{2p_y},
\qquad \forall t\in \RR^\ast,
\]
which, combined with Lemma \ref{lemma:quick-lemma}, yields $(\ref{eqn:capa-1-psg})$.

Now, the capacity $\capa_\Q(-,\tup)$ is known to be continuous (see for example \cite{BezGauTsu-2025}) and, using $(\ref{eq:limit-1-psg})$ and $(\ref{eqn:capa-1-psg})$, we obtain
\[
\capa_\Q(\operatorname{diag}(A\cdot \V),\tup)
=
\lim_{t\to 0}\capa_\Q(\lambda(t)\cdot (A\cdot \V),\tup)
=
\capa_\Q(A\cdot \V,\tup).
\]
Applying Lemma \ref{lemma:quick-lemma} once more, we get that
\[
\capa_\Q(\V,\tup)
=
\capa_\Q(A^{-1}\cdot \operatorname{diag}(A\cdot \V),\tup)=\capa_\Q(\V', \tup),
\]
which finishes the proof.
\end{proof}

Although not needed for our main results, the following corollary recovers \cite[Theorem 2.3]{BezGauTsu-2025} when $\tup$ is rational. Their proof builds on operator scalings \cite{GarGurOliWig-2017}. Our proof is based on Theorem \ref{thm:semi-stable-summary} and the simple fact that any $\tup$-semi-stable representation admits a Jordan--H\"older filtration with $\tup$-stable factors.

\begin{corollary} For any feasible datum $(\V, \tup)$ with $\V \in \rep(\Q, \dd)$, there exists a geometric datum $(W, \tup)$ and a $1$-parameter subgroup $\widetilde{\lambda}:\RR^\ast \to \GL(\dd)$ such that
$$
\lim_{t\to 0}\widetilde{\lambda}(t)\cdot \V
= W.
$$
In particular, for every $\varepsilon>0$, there exists $U \in \GL(\dd)\V$ such that
$$
||U-W||_F<\varepsilon.
$$
\end{corollary}

\begin{proof} It follows from $(\ref{eq:limit-1-psg})$ that there exist an element $A \in \GL(\dd)$ and a $1$-parameter subgroup $\lambda$ such that 
$$
\lim_{t\to 0}\lambda(t)\cdot (A \cdot \V)
$$ 
exists; moreover, if $\widetilde{V}$ denotes this limit, then $(\widetilde{V}, \tup)$ is extremizable. 

By Corollary \ref{coro:extremizable-geometric}, there exists a $B \in \GL(\dd)$ such that $(B \cdot \widetilde{V}, \tup)$ is geometric. Now define a new $1$-parameter subgroup $\widetilde{\lambda}$ by
$$
\widetilde{\lambda}(t)_x:=B_x\cdot \lambda(t)_x \cdot A_x \in \GL(\dd_x) \qquad \text{for every } t \in \RR^\ast \text{~and~} x  \in \Q_0.
$$
It follows that
$$
\lim_{t\to 0}\widetilde{\lambda}(t)\cdot \V
= W,
$$
where $W:=B \cdot \widetilde{V}$. The proof now follows.
\end{proof}

We are now ready to prove Theorem \ref{thm:main}.

\begin{proof}[Proof of Theorem \ref{thm:main}] Let $\mathcal{Y} \subseteq \RR^{D_1}\times \RR^{D_1}$ be the semi-algebraic family from Theorem \ref{thm:partition-semi-alg}. The Definable Choice Property yields a semi-algebraic function $f:\s\to \RR^{D_1}$ such that 
$$
(\V, f(\V)) \in \mathcal{Y} \qquad \text{for all } \V \in \s.
$$
Furthermore, $f(\s)\subseteq \locus$ and $\capa_\Q(\V, \tup)=\capa_\Q(f(\V), \tup)$ for every $\V \in \s$ by Theorem \ref{thm:partition-semi-alg}{(2)}. Thus restricting the capacity to $f(\s)$, we get that 
\begin{equation}\label{eqn:cap-restricted}
\V \to \capa_Q(\V, \tup) \qquad (\V \in \s)
\end{equation}
is a semi-algebraic, and hence algebraic, function on $\s$ by Theorem \ref{thm:main-2}.
\end{proof}

\subsection*{Acknowledgment} We are grateful to Saugata Basu for bringing the Definable Choice Property to our attention, which was essential in completing the proof of our main result. We are also indebted to Neal Bez, Jon Bennett, and Shohei Nakamura for pointing out the necessity of the rationality assumption on $\tup$ in Theorem \ref{thm:main}, and for providing the construction in Example \ref{non-algebraic-example}.

C. Chindris is supported by Simons Foundation grant $\# 711639$. H. Derksen is supported by a Simons Fellowship.

\end{document}